\documentclass[11pt,reqno]{amsart}
\usepackage{setspace}
\usepackage{amsmath}
\usepackage{amscd,amsthm,amsfonts,amsopn,amssymb,verbatim}
\usepackage{mathrsfs}
\numberwithin{equation}{section}

\newtheorem{theorem}{Theorem} %[section]

\newtheorem{proposition}[theorem]{Proposition}

\newtheorem{lemma}{Lemma}

\theoremstyle{remark}

\DeclareMathOperator{\supp}{supp\,}

\def\be{\begin{equation}}
\def\ee{\end{equation}}

\def\ve{\varepsilon}

\allowdisplaybreaks
\begin{document}

\title[An application of group expansion]
{An application of group expansion to the Anderson-Bernoulli model}
\author
{J.~Bourgain}
\address
{Institute for Advanced Study, Princeton, NJ 08540}
\email
{bourgain@ias.edu}
\date{\today}

\begin{abstract}
We establish smoothness of the density of states for 1D lattice Schr\"odinger operators with potential taking values $\pm \lambda$,
for $\lambda$ in a class of small algebraic numbers and energy $E\in )-2, 2($ suitably restricted away from $\pm2$.
\end{abstract}
\maketitle

\section*
{\bf 0. Introduction}

Let $H=\Delta+\lambda V$, where $\Delta$ is the lattice Laplacian on $\mathbb Z$ and $V_z=(V_n)_{n\in\mathbb Z}$ are
independent variables in $\{1, -1\}$.
The spectral theory of this operator, referred to as the Anderson-Bernoulli model (A-B for short) has been studied
by various authors.
It was shown by Halperin \cite {S-T} that for fixed $\lambda$, the integrated density of states (IDS) $\mathcal N(E)$ of
$H$ is not H\"older continuous of any order $\alpha$ larger than
\be\label{0.1}
\alpha_0 =\frac {2\log 2} {\text {Arccosh $(1+\lambda)$}}.
\ee
H\"older regularity for some $\alpha>0$ has been established in several papers.

In \cite{Ca-K-M}, le Page's method is used.
Different approaches (including the super-symmetric formalism) appear in the paper \cite{S-V-W} that  relies on 
harmonic analysis principles around the uncertainty principle.
Recently \cite{B1}, the author showed that $\mathcal N(E)$ restricted to $\delta<|E|<2-\delta$ $(\delta>0$ fixed)
is at least H\"older-regular of exponent $\alpha(\lambda) \overset{\lambda\to 0}\rightarrow 1$.

It is believed that in fact for $\lambda\to 0$, $\mathcal N(E)$ becomes arbitrarily smooth and in particular
$\frac{d\mathcal N(E)}{dE}$ is bounded for $|\lambda|$ small enough.
No result of this type for the A-B model seems presently known.
Recall also Thouless formula relating  $\mathcal N(E)$ with the Lyapounov exponent $L(E)$ of $H$, i.e.
\be\label{0.2}
L(E)=\int\log |E-E'| d\mathcal N(E').
\ee
Since $\mathcal N(E)$ is obtained as the Hilbert transform of $L(E)$, their regularity properties may be derived
from each other.

The purpose of this Note is to prove the following in support of the above conjecture.
\medskip

\noindent
{\bf Theorem.}
{\it Let $H_\lambda$ be the A-B model considered above and restrict $|E|<2-\delta$ for some fixed $\delta>0$.
Given a constant $C>0$ and $k\in\mathbb Z_+$, there is some $\lambda_0=\lambda_0 (C, k)>0$ such that
$\mathcal N(E)$ is $C^k$-smooth on $]-2+\delta, 2-\delta[$ provided $\lambda$ satisfies the following 
conditions
\begin{itemize}
\item[(0.3)] \ $ |\lambda|<\lambda_0$
\item[(0.4)] \ $\lambda$ is an algebraic number of degree $d<C$ and minimal polynomial $P_d(x)\in\mathbb Z[X]$ with
coefficients bounded by $(\frac 1\lambda)^C$
\item[(0.5)] \ $\lambda$ has a conjugate $\lambda'$ of modulus $|\lambda'|\geq 1$
\end{itemize}
}

This seems in particular to be the first statement of Lipschitz behavior of the IDS for an A-B model.
Several comments are in order.
Firstly, the arithmetic assumptions on $\lambda$ permit to exploit a spectral gap theorem for the projective
action $\rho$ of $SL_2(\mathbb R)$ on $P_1(\mathbb R)$ that was established in \cite{B-Y} and which is our
main tool (cf. also the application in \cite{B2} of the latter result to regularity of Furstenberg measures).
This spectral gap property is not a consequence of hyperbolicity but is obtained by an adaptation to $SL_2(\mathbb R)$
of the arguments from \cite{B-G} on spectral gaps in $SU(2)$, established by methods from
arithmetic combinatorics (we will not elaborate on these aspects here; see also \S 4).
In its abstract setting, the result from \cite{B-Y} may be formulated as follows.
We identify $P_1(\mathbb R)$ with the torus $\mathbb T=\mathbb R/\mathbb Z$.

\begin{proposition}
\cite{B-Y}.

Given a constant $0<c<1$, there is $R_0\in\mathbb Z_+$ such that the following holds.
Let $ R> R_0$ and $\mathcal G\subset SL_2( R), |\mathcal G|= R$ generating
freely the free group $F_{R}$ on $ R$ generators.
Assume moreover

\begin{itemize}
\item [(0.6)] \ $\Vert g-e\Vert < R^{-c} $ for $g\in\mathcal G$
\item[(0.7)] \ $\mathcal G$ satisfies the following `non commutative diophantine condition'.
Denote $W_\ell(\mathcal G)\subset SL_2(\mathbb R)$ the set of words of length at most $\ell$
written in the $\mathcal G$-elements.  Then, for all $\ell \in\mathbb Z_+$
$$
\Vert g-e\Vert> R^{-\ell/c} \text { for } g\in W_\ell(\mathcal G)\backslash \{e\}.
$$
\end{itemize}

Then there is a finite dimensional subspace $V$ of $L^2(\mathbb T)$, that may be taken
$$
V=[e(n\theta); |n|<K] \quad (e(n\theta)= e^{2\pi i n\theta})
$$
where $K=K(R)\in\mathbb Z$ large enough,  such that if $f\in L^2(\mathbb T), \Vert f\Vert_2=1$ and
$f\bot V$, then
$$
\Big\Vert \frac 1{2R} \sum_{g\in\mathcal G}(\rho_gf +\rho_{g^{-1}} f)\Big\Vert_2
< \frac 12.\eqno{(0.8)}
$$
\end{proposition}
In the construction from \cite{B-Y}, the elements of $\mathcal G$ have rational entries, more
precisely, $\mathcal G\subset SL_2(\mathbb R)\cap \frac 1Q \text {Mat}_{2\times 2} (\mathbb Z)$ with
$Q\in \mathbb Z_+$ satisfying
\addtocounter{equation}{6}
\be\label{0.9}
Q^c <|\mathcal G|<R<Q.
\ee
Obviously $\Vert g-e\Vert\geq Q^{-\ell}$ for $g\in W_\ell(\mathcal G)\backslash \{e\}$ and in this way
we obtain condition (0.7).
In the application in this paper, $\mathcal G$ will consist of algebraic elements of bounded degree
$d<C$ and height bounded by $R^C$.
The required diophantine condition follows then from \cite{G-J-S}, Proposition 4.3, again invoking 
simple arithmetic considerations.
Presently, the \cite{G-J-S} argument seems the only known one to establish such non-commutative DC
and it is a major problem in this area of group expansion to treat non-algebraic generators.
This explains why in (0.4), $\lambda$ was assumed algebraic.
Let us next explain assumption (0.5), which in some sense is the novel input.
Denote for a fixed $E\in ]-2+\delta, 2-\delta[$
\be\label{0.10}
g_+=\begin{pmatrix} E+\lambda &-1\\ 1&0\end{pmatrix} \quad g_- =\begin{pmatrix} E-\lambda&-1\\ 1&0\end{pmatrix}.
\ee
Clearly
\begin{align}\label{0.11}
h_1& =g_+g_-^{-1} =\begin{pmatrix} 1&2\lambda\\ 0&1\end{pmatrix}\nonumber\\
h_2&= g_+^{-1} g_- =\begin{pmatrix} 1&0\\2\lambda&1\end{pmatrix}.
\end{align}

We use the following result due to Brenner \cite{Br}.

\begin{proposition}\label{Proposition2} (\cite{Br}).

If $\mu \in\mathbb R, |\mu|\geq 2$, then the group generated by the parabolic elements
$$
A=\begin{pmatrix} 1&\mu\\ 0&1\end{pmatrix} \text { and } \ B=\begin{pmatrix} 1&0\\ \mu&1\end{pmatrix}
$$
is free.
\end{proposition}

As pointed out in \cite{L-U}, the same conclusion holds if $\mu$ is an algebraic number with an algebraic conjugate $\mu'$ such that
$|\mu'|\geq 2$.
Hence, if $\lambda$ satisfies (0.5), the elements $h_1, h_2$ defined in \eqref{0.11} will generate a free group.
The set $\mathcal G$ in Proposition 1 is then obtained by considering elements $h_1^r h_2^r$, $r=1, \ldots, R$.
Using Proposition 1, we prove that
\be\label{0.12}
\Vert f- \rho_{g_+} f\Vert_2+ \Vert f-\rho_{g_-} f\Vert_2>\frac 18 \lambda^\tau
\ee
if $f\in L^2(\mathbb T), \Vert f\Vert_2=1, f\in V^\bot$.

Here $\tau>0$ is arbitrary and fixed, $|\lambda|$ taken sufficiently small depending on $\tau$ (for our purpose,
$\tau <\frac 12$ will do).
Note that the inequality \eqref{0.12}, restricted to $f\in V^\bot, \Vert f\Vert _2 =1$, is considerably stronger 
than the general inequality (cf. \cite{S-V-W}, Theorem 4.1)
\be\label{0.13}
\Vert f-\rho_{g_+} f\Vert_2 +\Vert f-f_{g_-} f\Vert_2> c|\lambda|
\ee
if $f\in L^2(\mathbb T), \Vert f\Vert_2=1$.

From \eqref{0.12}, we derive a restricted spectral gap for the operator \hfill\break
$\frac 14(I+\rho_{g_+}+\rho_{g_-})$. i. e.
\be\label{0.14}
\Big\Vert\frac 13 (f+\rho_{g_+} f+\rho_{g_-} f)\Big\Vert_2\leq (1-c\lambda^{2\tau}) \Vert f\Vert_2 \ 
\text { for } f\in V^\bot
\ee
and \eqref{0.14} is then processed further to derive certain smoothing estimates for the convolution powers (cf. \cite{B2}),
from which eventually the regularity of the Lyapounov exponent is derived.

Some comments about the energy restriction $|E|<2-\delta$.
At some stage of our analysis, we make use of the Figotin-Pastur transformation, setting
\be\label{0.15}
E=2\cos \kappa \ (0<\kappa< \pi)
\ee
and conjugating the cocycle by the matrix
\be\label{0.16}
S=\frac 1{(\sin \kappa)^{\frac 12}} \begin{pmatrix} 1&-\cos\kappa\\ 0&\sin\kappa\end{pmatrix}.
\ee
This gives
\be\label{0.17}
S g_{\pm} S^{-1} =\begin{pmatrix} \cos\kappa & -\sin\kappa\\ \sin \kappa& \cos\kappa\end{pmatrix}
\pm\lambda \begin{pmatrix} 1&\frac {\cos\kappa}{\sin\kappa}\\ 0&0\end{pmatrix}
\ee
which for small $\lambda$ are perturbations of a rotation.
We did not explore  here how to handle the edges of the spectrum.

Finally, let us point out that while $\lambda$ is taken small, we do not let $\lambda\to 0$ in the above
Theorem and the regularity estimates on $\mathcal N(E)$ degenerate in the limit $\lambda\to 0$. 

\section
{\bf A spectral gap estimate}

In this section, we prove the following

\begin{proposition}\label{Proposition3}
Fix constants $C>1, 0<\tau<\frac 12$. 
Let $\lambda$ be an algebraic number of degree $d<C$ and with minimal polynomial $P_d (x) =\sum^d_{j=0} a_j x^j \in \mathbb Z[X]$.
Assume
\begin{itemize}
\item[(1.1)] \ $|\lambda|, \lambda_0 =\lambda_0(C, \tau)<\frac 1{10}$
\item[(1.2)] \ $H=\max |a_j|<\big (\frac 1\lambda\big)^C$
\item[(1.3)] \ $\lambda$ has an algebraic conjugate $\lambda'$ with $|\lambda'|\geq 2$.
\end{itemize}

Denote

$$
h_1 =\begin{pmatrix} 1&\lambda\\ 0&1 \end{pmatrix}
\text { and } \ h_2=\begin{pmatrix} 1&0\\ \lambda&1\end{pmatrix}
$$
and let $\rho$ be the projective representation of $SL_2(\mathbb R)$, acting on $L^2(\mathbb T)$.
There is a finite dimensional space $V=[e(n\theta); |n|< K]$, where $K=K(\lambda)$, such that if
$f\in L^2(\mathbb T), \Vert f\Vert_2=1$ and $f\bot V$, then
$$
\Vert f-\rho_{h_1} f\Vert_2+\Vert f-\rho_{h_2} f\Vert_2 >\frac 14\lambda^\tau.
\eqno{(1.4)}
$$
\end{proposition}
\medskip

By (0.11), Proposition 3 implies (0.12) for $\lambda$ satisfying assumption (0.5) of the Theorem.

\medskip

\noindent
{\bf Proof of Proposition 3.}

The argument relies on Proposition 1 and 2 stated in Section 0.

Let $f$ be as above (with $K$ to be specified) and assume
$$
\Vert f-\rho_{h_1} f\Vert_2<\ve_0, \Vert f-\rho_{h_1} f\Vert_2 < \ve_0.\eqno{(1.5)}
$$
Denoting $W_\ell(h_1, h_2)$ the words of length at most $\ell$ written in $h_1, h_2$ and their inverses, it follows from (1.5)
that
$$
\Vert f-\rho_g f\Vert_2<\ell \ve_0 \text { for } \ g\in W_\ell(h_1, h_2).\eqno{(1.6)}
$$
By Proposition 2 and (1.3), $h_1, h_2$ are generators of the free group $F_2$.
Let
$$
R=[|\lambda|^{-\tau}]\eqno{(1.7)}
$$
and define for $r=1, \ldots, R$
$$
g_r = h_1^r h_2^r =\begin{pmatrix} 1&r\lambda\\ 0&1\end{pmatrix} \ \begin{pmatrix} 1&0\\ r\lambda&1\end{pmatrix}.
\eqno{(1.8)}
$$
Then $\mathcal G=\{ g_1, \ldots, g_R\}$ are free generators of $F_R$ and clearly satisfy
$$
\Vert 1-g\Vert<\lambda^{\frac 12} \text { for } \ g\in\mathcal G.\eqno{(1.9)}
$$
In order to apply Proposition 1, we need to verify the DC (0.7).
This is basically Proposition 4.3 from \cite{G-J-S}, but we recall the argument since the quantitative aspects of the
estimate matter here.

Take $N\in\mathbb Z_+$, $N\leq H$ such that $N\lambda =\mu\in\mathcal O =\mathcal O_{\mathbb Q(\lambda)}$ (the integers of
the number field $\mathbb Q(\lambda)$).
If $w\in W_\ell(\mathcal G)$, the entries of $w-1$ are, by (1.8), of the form $f(\lambda)$ with $f(x)\in \mathbb Z[X]$
of degree $D\leq 2\ell$ and coefficients bounded by $(2+R)^{2\ell}$.
Let $\lambda=\lambda_1, \lambda_2, \ldots, \lambda_d$ be the conjugates of $\lambda$ and set $\mu_j= N\lambda_j$ $(1\leq j\leq d)$
which are the conjugates of $\mu$.
Thus $N^Df(\lambda_j)=f_1(\mu_j)$ where $f_1(X)=N^d f\big(\frac XN\big) \in\mathbb Z[X]$.
Assuming $f(\lambda)\not= 0$, it follows that $\prod^d_{j=1} f_1(\mu_j)\in \mathbb Z\backslash \{0\}$ and hence
$$
|f_1 (\mu)|\geq N^{-(d-1)D}\prod^d_{j=2} |f(\lambda_j)|^{-1}.\eqno{(1.10)}
$$
Since $|\lambda_j| \leq H+1, |f(\lambda_j)|\leq (2+R)^{2\ell} (H+1)^{2\ell}$ and by (1.10), (1.7), (1.2)
$$
\Vert w-1\Vert\geq |f(\lambda)|\geq N^{-dD} [(2+R)(1+H)]^{-2\ell(d-1)}> R^{-4(\frac C\tau+1)d\ell}=R^{-C'\ell}
$$

Taking $|\lambda|< \lambda_0(C, \tau)$, we get $R>R_0$ and the conclusion of Proposition 1 applies with some $K$ depending
on the size of $\lambda$.

From (0.8), it follows in particular that for some $g\in\mathcal G\subset W_{2R}(h_1, h_2)$
$$
\frac 12 <\Vert f-\rho_g f\Vert_2 < 2R\ve_0
$$
implying (1.4).
This proves Proposition 3.\hfill$\square$
\medskip

In the sequel, we will use (0.12) for some fixed $\tau<\frac 12$. 

\section
{\bf Smoothing estimates}

For $g\in SL_2(\mathbb R)$, denote by $\tau_g$ the action on $\mathcal P_1(\mathbb R)$, identified  with the circle
$\mathbb T=\mathbb R/\mathbb Z$.
Thus if $g=\begin{pmatrix} a&b\\ c&d\end{pmatrix}, ad-bc=1$, then
$$
e^{i\tau_g(\theta)}=\frac {(a\cos \theta+b\sin\theta)+ i(c\cos\theta+d\sin\theta)}{[(a\cos\theta+ b\sin\theta)^2 +(c\cos\theta+
d\sin\theta)^2]^{\frac 12}}\eqno{(2.1)}
$$
and $\rho_gf=(\tau_{g^{-1}}')^{\frac 12}(f\circ \tau_{g^{-1}})$. Recall that
$$
\tau_g'(\theta) =\frac {\sin^2\tau_g(\theta)}{(c\cos \theta+ d\sin\theta)^2} =\frac 1{(a\cos \theta+b\sin\theta)^2 +
(c \cos \theta +d\sin \theta)^2}\eqno{(2.2)}
$$
hence
$$
\Vert g\Vert^{-2}\leq \tau_g' \leq \Vert g\Vert^2 \text { and } \ |\tau_g^{(s)}|\leq c_s\Vert g\Vert^{2s}
\text{ for } s\in\mathbb Z_+.\eqno{(2.3)}
$$
Assume $|E|<2-\delta$ and perform the Figotin-Pastur transformation (0.15)-(0.17) denoting $\tilde g_\pm =
Sg_\pm S^{-1}$.
Since $\rho_{\tilde g_\pm}=\rho_S \rho_{g_\pm} \rho_{S^{-1}}$, it follows from (0.12)
that
$$
\Vert f-\rho_{\tilde g_+}f\Vert_2 +\Vert f-\rho_{\tilde g_-} f\Vert_2> \frac 18\lambda^\tau\eqno{(2.4)}
$$
provided $\Vert f\Vert_2 =1$, $\rho_{S^{-1}} f\in V^\bot$.
Since $\tau_S$ acts on $\mathbb T$ as a smooth diffeomorphism, the space $V$ may clearly be redefined as to ensure
that (2.4) holds for $f\in V^\bot$, $\Vert f\Vert_2=1$.
Observe also that by (0.17) and our assumption $|E|< 2-\delta$, $\delta$ fixed, $\tilde g_\pm$ are $O(\lambda)$ perturbations of 
a circle rotation.
Hence, by (2.2)
$$
\Vert\tilde g_\pm\Vert< 1+C\lambda\eqno{(2.5)}
$$
$$
\tau_{\tilde g_\pm}' = 1+O(\lambda).\eqno{(2.6)}
$$
Denoting
$$
\tilde T_1 =\frac 13 (I+\rho_{(\tilde g_+)^{-1}}+ \rho_{(\tilde g_-)^{-1}})\eqno{(2.7)}
$$
(2.4) implies that
$$
\Vert\tilde T_1f\Vert_2 < 1-\frac 1{2300} \lambda^{2\tau} \text { if } \ f\in V^\bot, \Vert f\Vert_2=1.\eqno{(2.8)}
$$
Since $\rho_{(\tilde g_\pm)^{-1}} f=\big((\tau_{\tilde g_\pm})'\big) ^{\frac 12} (f\circ \tau_{\tilde g_\pm})$, (2.6) clearly implies
(assuming $\lambda$ small enough)
$$
\Vert\tilde T f\Vert_2\leq \Big( 1-\frac 1{2301} \lambda^{2\tau}\Big)\Vert f\Vert_2 \text { for } \ f\in V^\bot\eqno{(2.9)}
$$
where $V=[e(n\theta); |n|< K]$ and we defined
$$
\tilde T f=\frac 13 \big(f+(f\circ\tau_{\tilde g_+})+(f\circ\tau_{\tilde g_-})\big).\eqno{(2.10)}
$$
For simplicity, we drop the $\sim$ notation in the next considerations.

Our next goal is to deduce from the contractive estimate (2.9) further bounds on $T^m$ acting on various spaces.
Note that obviously
$$
\Vert T^m f\Vert_\infty \leq \Vert f\Vert_\infty.\eqno{(2.11)}
$$
Let $g\in W_\ell (g_+, g_-), n\in\mathbb Z, n'\in\mathbb Z_*$.
By change of variable and partial integration, we obtain
$$
\begin{aligned}\label{2.12}
\Big|\int e(n' \tau_g (x)+nx)dx\Big|&=\Big|\int e\Big(n'y+n\tau_{g^{-1}} (y)\Big) \tau_{g^{-1}}(y)dy\Big|\nonumber\\
&\ll_r \ \frac 1{|n'|^r} \Vert e(n\tau_{g^{-1}}) \tau_{g^{-1}} ' \Vert_{C^r}\nonumber\\
& \ll_r \ \frac 1{|n'|^r} (|n|^r \Vert g\Vert^{2(r+1)}) \quad \big(\text { by (2.3)}\big)\nonumber\\
\end{aligned}
$$
$$
\qquad\qquad\ll_r \ \frac {|n|^r}{|n'|^r} (1+C|\lambda|)^{2(r+1)\ell}
\eqno{(2.12)}
$$
since $\Vert g\Vert < (1+C\lambda)^\ell $ from (2.5).

\begin {lemma}\label{Lemma1}
$$
\Vert T^m f\Vert_2\leq C(\lambda) \Vert f\Vert_2. \eqno{(2.13)}
$$
\end{lemma}

\begin{proof}

Denote $P_K$ the orthogonal (= Fourier) projection on $V$ and decompose $f=f^{(1)} + f^{(2)}, f^{(1)} =P_K f, f^{(2)}\bot V$.

Thus
$$
\Vert f^{(1)} \Vert_\infty \leq \sqrt{2K} \Vert f\Vert_2 \text { and } \Vert f^{(2)}\Vert_2 \leq \Vert f\Vert_2
$$
and
$$
\begin{aligned}
\Vert T^mf\Vert_2 &\leq \Vert T^m f^{(1)} \Vert_2 +\Vert T^m f^{(2)} \Vert_2\\
&\leq \Vert T^mf^{(1)} \Vert_\infty + \Vert T^{m-1} f_1\Vert_2\qquad (f_1= Tf^{(2)})\\
&\leq \Vert f^{(1)} \Vert_\infty +\Vert T^{m-1} f_1\Vert_2 \qquad \text {\big(by (2.11)}\big)\\
&\leq \sqrt{2K} \Vert f\Vert_2 +\Vert T^{m-1} f_1\Vert_2 \qquad (2.14)
\end{aligned}
$$
where, by (2.9),
$$
\Vert f_1\Vert_2 \leq (1- c\lambda^{2\tau})\Vert f^{(2)}\Vert_2\leq (1-c\lambda^{2\tau})\Vert f\Vert_2.
$$
Repeat (2.14) with $f$ replaced by $f_1$ and iterate to get
$$
\Vert T^m f\Vert_2 \lesssim \sqrt{2K} \lambda^{-2\tau}\Vert f\Vert_2
$$
proving (2.13).
\end{proof}

There is the following refinement of Lemma 1.

\begin{lemma}\label{Lemma2}
Let $\supp \hat f\cap [-2^k, 2^k]=\phi$ with $k>k(\lambda)$.

Then
$$
\Vert T^mf\Vert_2 \leq C(\lambda) e^{-\min(c\lambda^{2\tau} m, rk)} \Vert f\Vert_2\eqno{(2.15)}
$$
for any given $r\geq 1$ (assuming $\lambda$ small enough).
\end{lemma}

\begin{proof}
In view of Lemma 1, it suffices to establish (2.15) for $m<C\lambda^{-2\tau} rk$.

Set $F_m=T^mf$ and decompose $F_m =P_KF_m+(F_m-P_K F_m)=F_m^{(1)}+ F_m^{(2)}$.
Then, using (2.12)
$$
\begin{aligned}
|\hat F_m(n)|& \, \leq \max_{g\in W_m} \sum_{|n'|> 2^k}|\hat f(n')| \ |\widehat{e(n'\tau_g)} (n)|\\
&\, \ll_r |n|^r e^{C|\lambda|rm} \sum_{|n'|> 2^k} |\hat f(n')| \ |n'|^{-r}\\
& \, \ll_r |n|^r e^{Cr|\lambda|m} \  2^{-k(r-\frac 12)} \Vert f\Vert_2\\
\end{aligned}
$$
$$
\qquad\qquad\qquad\qquad\qquad
 \ll_r |n|^r\Big(e^{Cr|\lambda|^{1-2\tau}}\frac 1{\sqrt 2}\Big)^{rk} \Vert f\Vert_2 <|n|^r e^{-\frac 1{10} rk}\Vert f\Vert_2 \ \eqno{(2.16)}
$$
by the assumption on $m$ and $\lambda$ sufficiently small $(\tau<\frac 12)$.

Thus
$$
\Vert F_m^{(1)} \Vert_\infty\leq \sqrt{2K} \Vert F_m^{(1)}\Vert_2\leq CK^{r+ 1} \, e^{-\frac 1{10} rk}\Vert f\Vert_2.
$$
Estimate
\begin{align}
\Vert F_{m+1} \Vert_2 &\leq \Vert TF_{m}^{(1)}\Vert_\infty +\Vert TF_m^{(2)}\Vert_2\nonumber\\
&\leq CK^{r+1} e^{-\frac 1{10} rk}\Vert f\Vert_2 +(1-c\lambda^{2\tau}) \Vert F_m\Vert_2\tag{2.17}
\end{align}
where we used again (2.9).

Iteration of (2.17) with $m<Cr\lambda^{-2\tau}k$ gives
$$
\Vert F_m\Vert_2 \leq [C_r(\lambda) e^{-\frac 1{10} rk}+ e^{-c\lambda^{2\tau}}m]\Vert f\Vert_2.
$$
This proves (2.15).
\end{proof}

Next, we establish bounds on higher Sobolev norms.

\begin{lemma}\label{Lemma3}
For $s\in \mathbb Z_+$, $|\lambda|<\lambda(s)$, we have for $f\in H^s(\mathbb T)$
$$
\Vert T^m f\Vert_{H^s} \leq C(\lambda)\Vert f\Vert_2 + e^{-c(\lambda)m}\Vert f\Vert_{H^s}.\eqno{(2.18)}
$$
In particular
$$
\Vert T^m f\Vert_{H^s}\leq C\Vert f\Vert_{H^s}.
$$
\end{lemma}

\begin{proof}

Apply Lemma 2 with $m=m_0(\lambda)$ to specify, $K_1 =2^{m_0}$, to obtain
$$
\Vert T^{m_0} (I-P_{K_1})\Vert_{2\to 2} \leq C(\lambda) e^{-c\lambda^{2\tau}m_0}\eqno{(2.20)}
$$
while on the other hand for $s\in\mathbb Z_+$
\begin{align}
\Vert T^{m_0}(I-P_{K_1})\Vert_{H^s\to H^s} \leq \Vert T^{m_0}\Vert_{H^s\to H^s} &< C_s\max_{g\in W_{m_0}}\Vert g\Vert^{2s}\nonumber\\
&< C_s e^{C\lambda sm_0}.\tag{2.21}
\end{align}

Assuming $\lambda$ sufficiently small and taking $m_0=m_0(\lambda, s)$, interpolation between (2.20), (2.21) will imply that
$$
\Vert T^{m_0} (I-P_{K_1})\Vert_{H^s\to H^s}<\frac 1{10}.\eqno{(2.22)}
$$
Set $F_m=T^mf$.
Then
$$
\begin{aligned}
\Vert F_{m+m_0}\Vert_{H^s}&\leq \Vert T^{m_0} P_{K_1} F_m\Vert_{H^s}+\Vert T^{m_0} (I-P_{K_1}) F_m\Vert_{H^s}\\
&\overset{(2.22)}\leq C(\lambda)K^s_1 \Vert F_m\Vert_2 +\frac 1{10} \Vert F_m\Vert_{H^s}\\
\end{aligned}
$$
$$
\overset{(2.13)} \leq C(\lambda)\Vert f\Vert_2 +\frac 1{10} \Vert F_m\Vert_{H^s}.\eqno{(2.23)}
$$
Iteration of (2.23) implies (2.18).
\end{proof}

Lemmas 1, 2, 3 hold for $\tilde T$ defined in (2.10).
If we define now $T$ by
$$
Tf= \frac 13 \big(f+(f\circ \tau_{g_+}) +(f\circ \tau_{g-})\big)\eqno{(2.24)}
$$
clearly $T$ and $\tilde T$ are related by
$$
\tilde T f=\big(T(f\circ \tau_S)\big) \circ \tau_{S^{-1}}
$$
with $S$ given by (0.16).
Thus $\tau_S$ intertwines $T^m$ and $(\tilde T)^m$, Lemma 3 remains valid for the original $T$ given by (2.24).

Let $\mu$ be the probability measure on $SL_2(\mathbb R)$ defined by
$$
\mu=\frac 12(\delta_{g_+}+ \delta_{g_-}).\eqno{(2.25)}
$$
The Furstenberg measure $\nu$ is the (unique) $\mu$-stationary measure on $P_1(\mathbb R)\simeq \mathbb T$, i.e. satisfying
$$
\nu=\sum_g (\tau_g)_*[\nu]\mu(g).\eqno{(2.26)}
$$
For $f\in C^1(\mathbb T)$, one has large deviation inequalities (cf. \cite{B-L}) of the form
$$
\Big\Vert\sum_g(f\circ \tau_g)\mu^{(\ell)}(g) -\int fd\nu\Big\Vert_\infty \leq C e^{-c(\lambda)\ell} \Vert f\Vert_{C^1}.
\eqno{(2.27)}
$$
Since
$$
T=\frac 13 I+\frac 23\sum(\tau_g)_* \mu(g)
$$
$$
T^\ell=3^{-\ell}\sum^\ell_{m=0} \begin{pmatrix} \ell\\ m\end{pmatrix} 2^m\Big(\sum (\tau_g)_*\mu^{(m)}(g)\Big).\eqno{(2.28)}
$$
\medskip

Combined with (2.27), this gives

\begin{lemma}\label{Lemma4}
$$
\Vert T^\ell f-\int fd\nu\Vert_\infty \leq C(\lambda) e^{-c(\lambda)\ell}\Vert f\Vert_{C^1}.\eqno{(2.29)}
$$
\end{lemma}

\begin{proof}
L.h.s. of (2.29) is bounded by
$$
C\Vert f\Vert_{C^1} 3^{-\ell} \sum^\ell_{m=0} \begin{pmatrix} \ell\\ m\end{pmatrix} 2^m e^{-c(\lambda)m}
< C\Vert f\Vert_{C^1} \Big(\frac 23+\frac 13 \ e^{-c(\lambda)}\Big)^\ell.
$$
\end{proof}

\begin{lemma}\label{Lemma5}
For $s \geq 1$ and $f\in H^{s+1}$
$$
\Vert (T^\ell f)' \Vert_{H^s} \leq C(\lambda) e^{-c(\lambda)\ell} \Vert f\Vert_{H^{s+1}}.
$$
\end{lemma}

\begin{proof}
Choose some $\ell_1<\ell$ and write
$$
\begin{aligned}
\Vert (T^\ell f)' \Vert_{H^s}&\leq \Vert T^\ell f-\int fd\nu\Vert_{H^{s+1}}\\
&\leq \Vert T^{\ell_1} (T^{\ell-\ell_1} f-\int fd\nu)\Vert_{H^{s+1}}\\
&\leq C(\lambda)\Vert T^{\ell-\ell_1} f-\int fd\nu\Vert_2+ e^{-c(\lambda)\ell_1} \Vert T^{\ell -\ell_1} f\Vert_{H^{s+1}}
\text { (by Lemma 3)}\\
&\leq C(\lambda) e^{-c(\lambda) (\ell-\ell_1)}\Vert f\Vert_{C^1} + C(\lambda)e^{-c(\lambda)\ell_1}\Vert f\Vert_{H^{s+1}} 
\text{ (by Lemmas 4, 3)}.
\end{aligned}
$$
and (2.30) follows by taking $\ell_1\sim \frac \ell 2$.
\end{proof}

\section
{\bf Smoothness of Lyapounov exponent and density of states}

Recall Thouless' formula
$$
L(E)=\int \log |E-E'| d\mathcal N(E')
$$
which shows that the Lyapounov exponent $L(E)$ and the IDS $\mathcal N(E)$ are related by the Hilbert transform.
Hence it suffices to consider smoothness of $L(E)$.

Recall also that if $\eta$ is the site distribution of $H$, then
\begin{align}
L(E)&= \iint \log \Big\Vert\begin{pmatrix} E-V&-1\\ 1&0\end{pmatrix} \begin{pmatrix} \cos\theta\\ \sin\theta\end{pmatrix}
\Big\Vert \eta(dv) \nu_E(d\theta)\nonumber\\
&= \int \underset\pm {Av} \log\Big\Vert\begin{pmatrix} E\pm\lambda&-1\\ 1&0\end{pmatrix} \begin{pmatrix} \cos\theta\\ \sin\theta
\end{pmatrix} \Big\Vert \nu_E(d\theta)
\end{align}
in the Bernoulli case.  Denote
\be\label{3.2}
\Phi_E(\theta) = \underset{\pm}{Av} \log\Big\Vert\begin{pmatrix} E\pm\lambda& -1\\ 1&0\end{pmatrix} \begin{pmatrix} \cos\theta\\ \sin\theta
\end{pmatrix}\Big\Vert%\eqno{3.2}
\ee
which is a smooth function in $(\theta, E)$.

By (3.1) and Lemma 4,
\be\label{3.3}
\Vert L(E)-(T_E)^\ell \Phi_E\Vert_\infty < C e^{-c\ell}%\eqno{(3.3)}
\ee
noting the dependence of $T$ on $E$ (constants in the sequel may depend on $\lambda$).

\medskip

\noindent
{\bf Proof of the Theorem.}

By the preceding, it suffices to show that $L(E)$ is a $C^k$-function of $E$, assuming $\lambda_0$ in (0.3) sufficiently small.

By (3.3), it will suffice to establish bounds on $\partial_E^{(k)} (T_E^{\ell}\Phi_E)$ that are uniform in $\ell$.

Returning to (0.10), let $\mathcal G=\{g_+^{(E)}, g_-^{(E)}, 1\}$.
For $g_1, \ldots, g_\ell \in \mathcal G$, the chain rule gives
\begin{align}\label{3.4}
&\partial _E(\Phi_E \circ \tau_{g_1\ldots g_\ell})=\nonumber\\
&(\partial_E \Phi_E)\circ \tau_{g_1\ldots g_\ell}+\nonumber\\
&\sum^\ell_{m=1} [(\Phi_E\circ\tau_{g_1\ldots g_{{m-1}}})' \circ \tau_{g_{m\ldots g_\ell}}][(\partial_E \tau_{g_m})
\circ \tau_{g_{m+1}} \ldots g_\ell]
\end{align}
where $\partial_E\tau_g=- \sin^2\tau_g$.
Averaging (3.4) gives therefore
\begin{align}\label{3.5}
\partial_E(T^\ell_E \Phi_E)&= T^\ell (\partial _E\Phi_E)\nonumber\\
&-\sum^\ell_{m=1} T^{\ell-m+1} [(T^{m-1} \Phi_E)' \sin^2\theta].
\end{align}
Thus
$$
|(3.5)|<C+\sum^\ell_{m=1} \Vert (T^{m-1} \Phi_E)'\Vert_\infty
$$
and applying Lemma 5 with $f=\Phi_E$ and $s=1$ shows that $\Vert (T^m\Phi_E)' \Vert_\infty \leq C e^{-cm}$.

For $s=2$, one obtains by iteration of \eqref{3.5} expansions of the form
\be\label{3.6}
T^{m_1} \big(\sin^2\theta (T^{m_2} \big(\sin^2\theta(T^{m_3}\Phi_E)'))'\big)
\ee
where $\ell =m_1 +m_2+m_3$.

Again from Lemma 5, applied consecutively for $s=1, s=2$,
\begin{align}
|\eqref{3.6}|&\lesssim \Vert \big(T^{m_2} \big(\sin^2\theta (T^{m_3} \Phi_E)')\big)'\Vert_ {H^1}\nonumber\\
&\lesssim e^{-cm_2} \Vert (T^{m_3}\Phi_E)'\Vert_{H^2}\nonumber\\
&\lesssim e^{-c(m_2+m_3)}.\nonumber
\end{align}
The continuation of the process is clear.

\medskip
\section
{\bf Further comments}

1. \ One could conjecture a restricted spectral gap of the form (0.12) to be valid without arithmetical assumptions on $\lambda$.
This would enable us to show that the density of states of the A-B model is $C^k$-smooth provided the coupling $\lambda\not=0$
is sufficiently small (at least with $E$ restricted as in the above Theorem).
Note that algebraic hypothesis on $\lambda$ appear in two places.
Firstly in the expansion result from \cite{B-Y}, where it is used to establish the non-commutative diophantine property
of the group (see also \cite {B-G}).
In fact weaker properties (such as positive box dimension at appropriate scales) would suffice.
But the only available technique so far is that from \cite{G-J-S} using arithmetic heights.
Secondly, our application of Brenner's result is based on algebraic conjugation.
The conclusion from Proposition 2 is known to fail for certain values of $\mu$ and a complete understanding of
which are the `free' values of $\mu$ seems not available at the present.

2. \ The A-B model may in some sense be viewed as a non-commutative version of the classical Bernoulli convolution
problem about which there is an extensive literature.
Recall that for $0<\lambda<1$, one considers the measure $\nu_\lambda$ obtained from the random series
$$
\sum^\infty_{n=0} v_n \lambda^n\eqno {(4.1)}
$$
where $\{v_n\}$ is a sequence of independent $\pm 1$-valued Bernoulli variables, $\mathbb P(v_n=1)=\mathbb P(v_n+-1)=\frac 12$.
As pointed out in \cite {L-V}, $\nu_\lambda$ is $\mu_\lambda$-stationary, where $\mu_\lambda$ is the probability measure
supported on the two similarities $x\to \lambda x\pm 1$ putting 1/2 mass on each.
A major problem about the measures $\nu_\lambda$ is their absolute continuity.
Starting from the work of Erd\"os, several results on this issue were obtained.
In particular Solomyak \cite {Sol} proved that $\nu_\lambda$ is absolutely continuous for almost all $\lambda>\frac 12$,
while Erd\"os observed that $\nu_\lambda$ is singular if $\lambda^{-1}$ is a Pisot number.
Returning to the A-B model, the situation turns out to be quite different, as our Theorem applies in particular
if $\lambda^{-1}$ is a sufficiently large Pisot number and in this case the Furstenberg measure is absolutely
continuous with $C^k$-density.
The latter statement follows easily from the above analysis indeed (cf. also \cite{B2}).
Let $f\in L^\infty(\mathbb T)$, $|f|\leq 1$ and supp\,$\hat f\subset [-2^{k+1}, -2^k]\cup [2^k, 2^{k+1}]$.
By (2.26), (2.28), $\langle \nu, f\rangle =\langle\nu, T^mf\rangle$ for all $m$.
Taking now $m$ large enough and applying the above Lemmas 3 and 2, it follows that
$$
\Vert T^{2m}f\Vert_\infty \leq C\Vert T^{2m} f\Vert_{H^1}\leq C\Vert T^m f\Vert_1\leq e^{-rk}< C_\lambda^{-k}\eqno {(4.2)}
$$
where $C_\lambda$ can be made arbitrarily large for $\lambda$ small enough.
Hence we obtain $|\langle\nu, f\rangle|< C^{-k}_\lambda$, from where the smoothness claim for $\frac {d\nu}{d\theta}$.

\medskip

\noindent
{\bf Acknowledgment:} The author is grateful to the mathematics
department of UC Berkeley for their hospitality.

\end{document}